\documentclass[a4paper, 12pt]{amsart}

\usepackage{amscd}
\usepackage{amssymb, color}

\theoremstyle{plain}
\newtheorem{theorem}{Theorem}[section]
\newtheorem{definition}{Definition}
\newtheorem{example}{Example}
\newtheorem{lemma}[theorem]{Lemma}
\newtheorem{corollary}[theorem]{Corollary}
\newtheorem{proposition}[theorem]{Proposition}

  \newtheorem{thqt}{Theorem}

\numberwithin{equation}{section}

\author[T. Yamazaki]{Takeaki Yamazaki}
\address{%
Department of Electrical, Electronic and Computer Engineering,
Toyo University,
Kawagoe-Shi, Saitama, 350-8585, Japan.
} \email{t-yamazaki@toyo.jp}



\title[A generalization of the Aluthge transformation]
{A generalization of the Aluthge transformation
in the viewpoint of operator means}

\thanks{This research is supported by 
the INOUE ENRYO Memorial Grant, TOYO University.}

\begin{document}

\begin{abstract}
The Aluthge transformation is generalized in the 
viewpoint of  
the axiom of operator means by
using double operator integrals. 
It includes the mean transformation which is 
defined by S. H. Lee, W. Y. Lee and Yoon.
Next we shall give some properties of it.
Especially, 
we shall show that the $n$-th iteration of 
mean transformation
of an invertible matrix converges to a normal matrix.
Inclusion relations among
numerical ranges of 
generalized Aluthge transformations 
respect to some operator means are
considered.
\end{abstract}

\maketitle

\section{Introduction}
Let $B(\mathcal{H})$ 
be the $C^{*}$-algebra of all 
bounded linear operators on a 
complex Hilbert space $\mathcal{H}$.
Let $T=U|T|\in B(\mathcal{H})$ be the polar 
decomposition of $T$.  
The {\bf Aluthge transformation} $\Delta(T)$ 
of $T$ is defined 
 in \cite{A1990} as follows.
\begin{equation}
\Delta(T):=|T|^{\frac{1}{2}}U|T|^{\frac{1}{2}}.
\label{defi:Aluthge transformation}
\end{equation}
Several properties of 
the Aluthge transformation has been studied, 
for example, 
(i) $\sigma(\Delta(T))=\sigma(T)$, 
where $\sigma(T)$ is the spectrum of 
$T\in B(\mathcal{H})$ \cite{H1997}, 
(ii) $\Delta(T)$ has a non-trivial invariant 
subspace if and only if $T$ does so \cite{JKP2000},
and (iii) if $T$ is semi-hyponormal (i.e., 
$|T^{*}|\leq |T|$), then 
$\Delta(T)$ is hyponormal (i.e., 
$\Delta(T)\Delta(T)^{*}\leq 
\Delta(T)^{*}\Delta(T)$) \cite{A1990},
where ``$\leq$'' means the Loewner partial order.
By considering the Loewner-Heinz inequality, 
hyponormality of an operator implies 
semi-hyponormality, but the 
converse implication does not hold in general.
Hence the Aluthge transformation
$\Delta(T)$ may have 
better properties than $T$.
Recently, a related new operator transformation 
has been defined in
\cite{LLY2014} and 
studied in\cite{arXivCCM2018, L2016} which is called the 
{\bf mean transformation} $\hat{T}$ of $T$. 
The definition is
$$ \hat{T}:=\frac{U|T|+|T|U}{2}. $$ 
The aim of this paper is to unify 
these operator transformations in the viewpoint of 
operator means, and give some properties.

\medskip

An operator mean is a binary operation 
on positive semi-definite operators. 
It was defined by Kubo-Ando as follows.
Let $B(\mathcal{H})^{+}$ and 
$B(\mathcal{H})^{++}$
be the sets of 
positive semi-definite and 
positive invertible operators, respectively.

\begin{definition}[Operator mean, \cite{KA1980}]
\label{defi: operator mean}
Let $\frak{M}: B(\mathcal{H})^{+}\times  B(\mathcal{H})^{+} 
\to B(\mathcal{H})^{+}$ be
a binary operation. If $\frak{M}$ satisfies the 
following four conditions, then $\frak{M}$ is called
an {\bf operator mean}.
\begin{itemize}
\item[(1)] If $A\leq C$ and $B\leq D$, then 
$\frak{M}(A,B)\leq \frak{M}(C,D)$,
\item[(2)] $X^{*}\frak{M}(A,B)X\leq 
\frak{M}(X^{*}AX, X^{*}BX)$ for all $X\in 
B(\mathcal{H})$,
\item[(3)] $A_{n}\searrow A$ and 
$B_{n}\searrow B$ imply $\frak{M} (A_{n}, B_{n})
\searrow \frak{M} (A,B)$ in the strong 
operator topology,
\item[(4)] $\frak{M}(I,I)=I$, where $I$ means the 
identity operator in $B(\mathcal{H})$.
\end{itemize}
\end{definition}

To get a concrete formula of an operator mean,
the following relation is very important.
Let $f$ be a  real-valued function defined on 
an interval $J\subseteq (0, \infty)$.
Then $f$ is said to be {\bf operator monotone} 
if $A\leq B$ for self-adjoint operators
$A,B\in B(\mathcal{H})$
whose spectra are contained in $J$, 
then $f(A)\leq f(B)$, where $f(A)$ and
$f(B)$ are defined by the functional calculus.

\begin{thqt}[\cite{KA1980}]
Let $\frak{M}$ be an operator mean.
Then there exists an operator monotone 
function $f$ on $(0, \infty)$ such that $f(1)=1$ and 
$$ \frak{M}(A,B)=A^{\frac{1}{2}}f(A^{-\frac{1}{2}}
BA^{-\frac{1}{2}})A^{\frac{1}{2}} $$
for all $A\in B(\mathcal{H})^{++}$ and 
$B\in B(\mathcal{H})^{+}$.
\end{thqt}
If $A\in B(\mathcal{H})^{+}$, we can obtain 
$\frak{M}(A,B)=\lim_{\varepsilon \searrow 0}
\frak{M}(A+\varepsilon I, B)$
because $A+\varepsilon I\in B(\mathcal{H})^{++}$
for $\varepsilon >0$ and Definition \ref{defi: operator mean}
(3). 
The function $f$ is called a {\bf representing 
function} of an operator mean $\frak{M}$.
Throughout this paper, we note 
$\frak{M}_{f}$ for an operator mean 
with a representing function $f$.
In this case, 
$f'(1)=\lambda \in [0,1]$ holds (cf. \cite{H2013}),
and we sometimes 
call $\frak{M}_{f}$ a $\lambda$-weighted 
operator mean. Moreover, 
if $\lambda = f'(1)\in (0,1)$, then
\begin{equation}
[1-\lambda +\lambda x^{-1}]^{-1}\leq
f(x)\leq 1-\lambda +\lambda x 
\label{eq:condition}
\end{equation}
holds for all $x>0$ (cf. \cite{P2016}). 

Typical examples of operator means 
are the $\lambda$-weighted geometric  
and $\lambda$-weighted power means. 
These representing functions are 
$f(x)=x^{\lambda}$ and 
$f(x)=[1-\lambda+\lambda x^{t}]^{\frac{1}{t}}$, 
respectively, 
where $\lambda\in [0,1]$
and $t\in[-1,1]$ (in the case $t=0$, we consider 
$t\to 0$). The weighted power mean interpolates 
arithmetic, geometric and harmonic means 
by putting $t=1,0, -1$, respectively.
%

The aim of this paper is to apply operator means 
to generalize the Aluthge transformation.
To do this, firstly, we shall explain how to generalize 
the Aluthge transformation in the
 matrices case in Section 2.
Then to extend the discussion in Section 2 into 
Hilbert space operators, 
we will introduce the double operator integrals
in Section 3.
We show that all
operator means can be applied to 
double operator integrals.
In Section 4, we shall generalize the Aluthge 
transformation respect to an arbitrary 
operator mean via double operator integrals.
In Sections 5 and 6, we shall give properties of 
the generalized Aluthge transformation.
In Section 5, we shall consider $n$-th iterated 
generalized Aluthge transformation. We divide this 
discussion into finite and infinite Hilbert space cases.
More precisely, we shall show that 
$n$-th iterated mean transformation 
of every invertible matrix converges to a 
normal matrix,
and show that there is a weighted unilateral shift 
on $\ell^{2}$ such that $n$-th iterated generalized 
Aluthge transformation
does not converge in a week operator topology.
In Section 6, we give inclusion relations among
numerical ranges of 
generalized Aluthge transformations.

\section{A generalization of the Aluthge transformation 
in the matrices case}

In this section, we shall generalize 
the Aluthge transformation in 
the matrices case which is a motivation of 
this paper.
Let $\mathcal{M}_{m}$ be a set of all $m$--by--$m$ 
matrices. 
It is known that $\mathcal{M}_{m}$ is a 
Hilbert space with an inner product
$\langle A,B\rangle :=
\mbox{trace}AB^{*}$.
For $A,B\in \mathcal{M}_{m}$, let 
$\mathbb{L}_{A}$ and
$\mathbb{R}_{B} $ be bounded linear operators on $\mathcal{M}_{m}$  
defined as follows:
$$ \mathbb{L}_A(X)=AX\quad 
\mbox{and}\quad \mathbb{R}_B(X)=XB $$
for $X\in \mathcal{M}_{m}$.
They are called the {\bf left and right multiplications}, 
respectively.
If $A,B\in \mathcal{M}_{m}$ are positive 
semi-definite (resp. positive invertible)
matrices, then $\mathbb{L}_{A}$ and
$\mathbb{R}_{B}$ are positive semi-definite 
(resp. positive invertible) operators
on $\mathcal{M}_{m}$, too.
It is easy to see that $\mathbb{L}_{A}$ and
$\mathbb{R}_{B}$ are commuting on the product, 
i.e., 
$$ \mathbb{L}_{A}\mathbb{R}_{B}(X)=
\mathbb{R}_{B}\mathbb{L}_{A}(X)=AXB $$
holds for all $X\in \mathcal{M}_{m}$.
Moreover, for each Hermitian 
$A\in \mathcal{M}_{m}$, 
$f(\mathbb{L}_{A})(X)=\mathbb{L}_{f(A)}(X)$
(resp. $f(\mathbb{R}_{A})(X)=\mathbb{R}_{f(A)}(X)$)
holds for all analytic functions $f$ 
if $f(A)$ is defined.
Hence we can consider operator means of 
$\mathbb{L}_{A}$ and $\mathbb{R}_{B}$.
For example, 
the arithmetic mean $\frak{A}$ of 
$\mathbb{L}_{A}$ and $\mathbb{R}_{B}$ 
is computed by
$$\frak{A}(\mathbb{L}_{A}, \mathbb{R}_{B})(X)=
\frac{
\mathbb{L}_{A}+\mathbb{R}_{B}}{2}(X)=
\frac{AX+XB}{2}. $$

\begin{theorem}\label{prop:matrix}
Let $A,B\in \mathcal{M}_{m}$ be positive 
invertible. Then for any operator mean 
$\frak{M}$, there exists a positive probability 
measure $d\mu$ on $[0,1]$ such that 
$$ \frak{M}(\mathbb{L}_{A}, \mathbb{R}_{B})(X)
=\int_{0}^{1}\left(\int_{0}^{\infty} 
e^{-x(1-\lambda)A^{-1}}X
e^{-x\lambda B^{-1}}dx\right)d\mu(\lambda) $$
for all $X\in \mathcal{M}_{m}$.
\end{theorem}

To prove Theorem 
\ref{prop:matrix},
we shall use the following result.

\begin{thqt}[\cite{H1951}, 
{\cite[Theorem VII.2.3]{B1997}}]\label{thqt:equation}
Let $A$ and $B$ be operators whose spectra are 
contained in the open right half-plane and 
left half-plane, respectively. Then 
the solution of the equation $AX-XB=Y$ can be 
expressed as 
$$ X=\int_{0}^{\infty}e^{-xA}Ye^{xB}dx. $$
\end{thqt}

\begin{proof}
[Proof of Theorem \ref{prop:matrix}]
Firstly, we shall show the case  
of $\lambda$-weighted harmonic mean
$\frak{M}_{f}$ for $\lambda\in [0,1]$, i.e.,
a representing function $f$ of $\frak{M}_{f}$ is 
$$ f(x)=[1-\lambda+\lambda x^{-1}]^{-1}. $$
In this case, the harmonic mean 
of $\mathbb{L}_{A}$ and $\mathbb{R}_{B}$ on 
$\mathcal{M}_{m}$ is 
$$ \frak{M}_{f}(\mathbb{L}_{A}, \mathbb{R}_{B})(X)=
[(1-\lambda)\mathbb{L}_{A}^{-1}+
\lambda \mathbb{R}_{B}^{-1}]^{-1}(X). $$
We notice that if $A$ and $B$ are positive 
invertible matrices, then $\mathbb{L}_{A}$ and
$\mathbb{R}_{B}$ are positive invertible, and hence
the above formula is well-defined.
Put $Y:=\frak{M}_{f}
(\mathbb{L}_{A}, \mathbb{R}_{B})(X)$. Then
it is a solution of a matrix equation
$$ [(1-\lambda)\mathbb{L}_{A}^{-1}+
\lambda \mathbb{R}_{B}^{-1}](Y)=X. $$
Thus for $X\in \mathcal{M}_{m}$,
we just have to give a solution $Y$ of the following 
matrix equation
$$ (1-\lambda)A^{-1}Y+\lambda YB^{-1}=X, $$
and it is equivalent to
$$ \{(1-\lambda)A^{-1}\}Y-Y(-\lambda B^{-1})=X. $$
By Theorem \ref{thqt:equation},
we have $Y=\frak{M}_{f}(\mathbb{L}_{A},
\mathbb{R}_{B})(X)$ as follows.
\begin{equation}
\begin{split}
\frak{M}_{f}(\mathbb{L}_{A},\mathbb{R}_{B})(X)
 & = \int_{0}^{\infty}e^{-x(1-\lambda)A^{-1}}
Xe^{-x\lambda B^{-1}}dx\\
 & = \int_{0}^{1}\left(
 \int_{0}^{\infty}e^{-x(1-\lambda)A^{-1}}
Xe^{-x\lambda B^{-1}}dx\right) d\mu(\lambda), 
\end{split}
\label{eq:matrix-harmonic mean}
\end{equation}
where $\mu ([0,1]):= \delta_{\{\lambda\}}([0,1])$,
the Dirac delta supported on $\{\lambda\}$.

Next we shall show an arbitrary 
operator mean case. Let $\frak{M}_{f}$ be 
an operator mean. Then it is known (cf. 
\cite{H2013}) that
there exists a 
positive probability measure $d\mu$ on 
$[0,1]$ such that
\begin{equation}
f(x)=\int_{0}^{1}[1-\lambda+\lambda x^{-1}]^{-1}
d\mu(\lambda). 
\label{eq:integral representation}
\end{equation}
Hence we have
\begin{align*}
\frak{M}_{f}(\mathbb{L}_{A},\mathbb{R}_{B})(X)
& =
\mathbb{L}_{A}
f(\mathbb{L}_{A}^{-1}\mathbb{R}_{B})(X)\\
& =
\int_{0}^{1} [(1-\lambda)\mathbb{L}_{A}^{-1}+
\lambda \mathbb{R}_{B}^{-1}]^{-1}d\mu(\lambda)(X)\\
& =
\int_{0}^{1} [(1-\lambda)\mathbb{L}_{A}^{-1}+
\lambda \mathbb{R}_{B}^{-1}]^{-1}(X)d\mu(\lambda)\\
& =
\int_{0}^{1} \left(
\int_{0}^{\infty}e^{-x(1-\lambda)A^{-1}}
Xe^{-x\lambda B^{-1}}dx\right)
d\mu(\lambda).
\end{align*}
\end{proof}

Now we shall give another formula of
$\frak{M}_{f}(\mathbb{L}_{A}, \mathbb{R}_{B})(X)$.

\begin{theorem}\label{thm: double summation}
Let $A, B\in \mathcal{M}_{m}$ be
positive invertible with the spectral
decompositions 
$A=\sum_{i=1}^{m}s_{i}P_{i}$ and 
$B=\sum_{j=1}^{m}t_{j}Q_{j}$, respectively.
Then for each operator mean $\frak{M}_{f}$, 
\begin{equation}
 \frak{M}_{f}(\mathbb{L}_{A}, \mathbb{R}_{B})(X)=
\sum_{i,j=1}^{m}
\mathcal{P}_{f}(s_{i}, t_{j})P_{i}XQ_{j}, 
\label{eq:double operator integral matrix}
\end{equation}
where the perspective $\mathcal{P}_{f}$ of $f$ is 
defined by
$ \mathcal{P}_{f}(s,t):=sf\left(\frac{t}{s}\right)$.
\end{theorem}

\begin{proof}
Let $A=\sum_{i=1}^{m}s_{i}P_{i}$ and 
$B=\sum_{j=1}^{m}t_{j}Q_{j}$ be spectral 
decompositions of $A$ and $B$, respectively.
For a representing function 
$f$ on $(0,\infty)$ of an operator mean $\frak{M}_{f}$, 
by \eqref{eq:integral representation},
the perspective $\mathcal{P}_{f}$ of $f$ is 
given by
$$ \mathcal{P}_{f}(s,t)
=\int_{0}^{1}[(1-\lambda)s^{-1}+\lambda t^{-1}]^{-1}
d\mu(\lambda) $$
for $s,t>0$. 
Then by Theorem \ref{prop:matrix}, we have
\begin{align*}
& \frak{M}_{f}(\mathbb{L}_{A}, \mathbb{R}_{B})(X)\\
& =\int_{0}^{1}\left( \int_{0}^{\infty} 
e^{-x(1-\lambda)A^{-1}}X
e^{-x\lambda B^{-1}}dx\right) d\mu(\lambda) \\
& 
=\int_{0}^{1}\left\{ \int_{0}^{\infty} 
\left(\sum_{i=1}^{m}
 e^{-x(1-\lambda)s_{i}^{-1}}P_{i}\right)
X
\left(\sum_{j=1}^{m}
e^{-x\lambda t_{j}^{-1}}Q_{j}\right)dx\right\}d\mu(\lambda) \\
& 
=\sum_{i,j=1}^{m} \int_{0}^{1}\left( \int_{0}^{\infty} 
e^{-\{(1-\lambda)s_{i}^{-1}+\lambda t_{j}^{-1}\}x}
dx\right) d\mu(\lambda)P_{i}XQ_{j}\\
& 
=\sum_{i,j=1}^{m} \int_{0}^{1} 
[(1-\lambda)s_{i}^{-1}+\lambda t_{j}^{-1}]^{-1}
d\mu(\lambda)P_{i}XQ_{j}   
=\sum_{i,j=1}^{m} 
\mathcal{P}_{f}(s_{i}, t_{j})P_{i}XQ_{j}.
\end{align*}
\end{proof}

If $A,B\in B(\mathcal{H})^{+}$, then for each 
$\varepsilon>0$, $A_{\varepsilon}:=A+\varepsilon I$
and $B_{\varepsilon}:=B+\varepsilon I$ are both 
positive invertible.
Then we can define 
$\frak{M}_{f}(\mathbb{L}_{A}, \mathbb{R}_{B})(X)$ by
$$ \frak{M}_{f}(\mathbb{L}_{A}, \mathbb{R}_{B})(X)=
\lim_{\varepsilon\searrow 0}
\frak{M}_{f}(\mathbb{L}_{A_{\varepsilon}}, 
\mathbb{R}_{B_{\varepsilon}})(X). $$

\begin{definition}
[Generalization of the Aluthge transformation]
\label{defi:matrix case}
Let $T=U|T|\in \mathcal{M}_{m}$ be the polar 
decomposition with the spectral decomposition 
$|T|=\sum_{i=1}^{n}s_{i}P_{i}$ of $|T|$.
For an operator mean $\frak{M}_{f}$, 
a {\bf generalization of the Aluthge transformation} 
$\Delta_{\frak{M}_{f}}(T)$ of $T$
respect to an operator mean 
$\frak{M}_{f}$ is defined by
$$ \Delta_{\frak{M}_{f}}(T):= 
\sum_{i,j=1}^{m} 
\mathcal{P}_{f}(s_{i}, s_{j})P_{i}UP_{j}. $$
\end{definition}

By using Theorem \ref{thm: double summation}, 
we have
another formula of a generalized Aluthge
transformation.

\begin{corollary}\label{cor: another formula}
Let $T=U|T|\in \mathcal{M}_{m}$ be  
the polar decomposition
such that $|T|$ is invertible. 
For an operator mean $\frak{M}$, 
there exists a positive probability measure 
$d\mu$ on $[0,1]$ such that
$$ \Delta_{\frak{M}}(T)=
\int_{0}^{1}\left( \int_{0}^{\infty} 
e^{-x(1-\lambda) |T|^{-1}}U
e^{-x\lambda |T|^{-1}}dx\right) d\mu(\lambda). $$
\end{corollary}

\begin{example}
Let $T\in \mathcal{M}_{m}$ with 
the polar decomposition $T=U|T|$ and
the spectral decomposition
$|T|=\sum_{i=1}^{m}s_{i}P_{i}$, and 
let $\frak{M}$ be an operator mean.
Then we have the following examples of 
$\Delta_{\frak{M}}(T)$.
Note that $\sum_{i=1}^{m} P_{i}=U^{*}U$.
\begin{itemize}
\item[(1)] Mean transformation 
\cite{LLY2014}.\\
Let $\frak{M}_{f}$ be the $\lambda$-
weighted arithmetic mean, i.e., 
the representing function of $\frak{M}_{f}$ is
$f(t)=1-\lambda +\lambda t$. Then
\begin{align*}
\Delta_{\frak{M}_{f}}(T)
& =
\sum_{i,j=1}^{m}
[(1-\lambda)s_{i}+\lambda s_{j}]
P_{i}UP_{j} \\
& =
(1-\lambda)\sum_{i,j=1}^{m}
s_{i}P_{i}UP_{j}+
\lambda \sum_{i,j=1}^{m}s_{j}P_{i}UP_{j}\\
& =
(1-\lambda)\left(\sum_{i,=1}^{m}s_{i}P_{i}\right)U
\left(\sum_{j=1}^{m}P_{j}\right)\\
& \hspace*{3cm} +
\lambda \left(\sum_{i=1}^{m}P_{i}\right)U
\left(\sum_{j=1}^{m}s_{j}P_{j}\right)\\
& = (1-\lambda)|T|U+\lambda U^{*}UU|T|.
\end{align*}
Especially, if $U$ is an isometry, then
$\Delta_{\frak{M}_{f}}(T)=\hat{T}$ 
(i.e., the weighted mean transform).
\item[(2)] Generalized Aluthge transformation
\cite{H1997}.\\
Let $\frak{M}_{f}$ be the $\lambda$-weighted 
geometric mean, i.e., the representing function of 
$\frak{M}_{f}$ is
$f(t)=t^{\lambda}$. Then 
\begin{align*}
\Delta_{\frak{M}_{f}}(T)
& =
\sum_{i,j=1}^{m} 
s_{i}^{1-\lambda}s_{j}^{\lambda}
P_{i}UP_{j} \\
& =
\left(\sum_{i=1}^{m} s_{i}^{1-\lambda} P_{i}\right)U
\left(\sum_{j=1}^{m} t_{j}^{\lambda} P_{j}\right)
 =
|T|^{1-\lambda}U|T|^{\lambda}.
\end{align*}
\end{itemize}
\end{example}

From Corollary \ref{cor: another formula}, we have
a basic property of a 
generalization of the Aluthge transformation.

\begin{theorem}\label{thm:fixed point}
Let $T\in \mathcal{M}_{m}$. Let 
$\frak{M}_{f}$ be an operator mean 
satisfying $f'(1)\in (0,1)$. 
Then  
$\Delta_{\frak{M}_{f}}(T)=T$ if and only if 
$T$ is normal.
\end{theorem}

\begin{proof}
For $\varepsilon >0$, $|T|_{\varepsilon}:=
|T|+\varepsilon I$ is positive invertible, and 
$|T|_{\varepsilon}\searrow |T|$ as $\varepsilon 
\searrow 0$. By considering this fact,
we may assume that $|T|$ is 
invertible.
Let $T=U|T|$ be the polar decomposition of $T$.
Then it is known that $T$ is normal if and only if 
$U|T|=|T|U$. Hence if $T$ is normal, then 
by Corollary \ref{cor: another formula},
\begin{align*}
\Delta_{\frak{M}}(T) & =
\int_{0}^{1}\left( \int_{0}^{\infty} 
e^{-x(1-\lambda) |T|^{-1}}U
e^{-x\lambda |T|^{-1}}dx\right) d\mu(\lambda)\\
& =
U \int_{0}^{1}\left( \int_{0}^{\infty} 
e^{- x|T|^{-1}}dx\right) d\mu(\lambda)=T.
\end{align*}

Let $|T|=\sum_{i=1}^{m}s_{i}P_{i}$
be the spectral decomposition.
Assume that $\Delta_{\frak{M}_{f}}(T)=T$ holds.
Then we have
$$  \sum_{i,j=1}^{m} 
\mathcal{P}_{f}(s_{i}, s_{j})P_{i}UP_{j}=
\Delta_{\frak{M}_{f}}(T)=T=
\sum_{j=1}^{m}s_{j}UP_{j}. $$
By multiplying $P_{i}$ and $P_{j}$ from the left and right
sides, respectively, 
$$   
\mathcal{P}_{f}(s_{i}, s_{j})P_{i}UP_{j}=s_{j}P_{i}UP_{j}, $$
since $P_{i}$ are orthogonal projections.
By $f'(1)\in (0,1)$ and an inequality \eqref{eq:condition}, 
$\mathcal{P}_{f}(s_{i}, s_{j})\neq s_{j}$ holds for $i\neq j$.
Then $P_{i}UP_{j}=0$ holds for $i\neq j$.
Hence we have
\begin{align*}
\Delta_{\frak{M}_{f}}(T) & =
\sum_{i,j=1}^{n} 
\mathcal{P}_{f}(s_{i}, s_{j})P_{i}UP_{j}\\
& =
\sum_{i=1}^{n} 
\mathcal{P}_{f}(s_{i}, s_{i})P_{i}UP_{i}\\
& =
\sum_{i=1}^{n} 
s_{i} P_{i}UP_{i}\\
& =
\sum_{i,j=1}^{n} 
\sqrt{s_{i}s_{j}}P_{i}UP_{j}=\Delta(T)
\quad \text{(Aluthge transformation).}
\end{align*}
Therefore $T$ is normal since $T=\Delta(T)$ 
\cite[Proposition 1.10]{JKP2000}.
\end{proof}

\section{Double operator integrals}
Although $\mathbb{L}_{A}$ and $\mathbb{R}_{B}$
can be defined on $B(\mathcal{H})$, 
we cannot consider operator means of 
$\mathbb{L}_{A}$ and $\mathbb{R}_{B}$,
since $B(\mathcal{H})$ is not a Hilbert space. 
To discuss similar argument in $B(\mathcal{H})$,
we shall use the double operator integrals. 
The double operator integrals
 was first appeared in \cite{DK1956}.
Then it was developed by 
Birman and Solomyak in \cite{BS1965}
and Peller in \cite{P1984, P1985}
(nice surveys are \cite{BS2003, HK2003}).
Let $A,B\in \mathcal{B(H)}^{+}$ with the 
spectral decompositions 
$$ A=\int_{\sigma(A)}sdE_{s}\text{ and }
B=\int_{\sigma(B)}tdF_{t}. $$
Let $\lambda$ (resp. $\mu$) be a finite 
positive measure on an interval 
$\sigma(A)$ (resp. $\sigma(B)$) equivalent  
(in the absolute continuity sense) to $dE_{s}$ 
(resp. $dF_{t}$). 
Let $\varphi\in 
L^{\infty}(\sigma(A)\times \sigma(B); \lambda\times 
\mu)$. For $X\in \mathcal{B(H)}$, 
the {\bf double operator
integrals} is given by 
$$ \Phi_{A,B,\varphi}(X) :=
\int_{\sigma(A)}\int_{\sigma(B)}
\varphi(s,t)dE_{s}XdF_{t}. $$
If $X\in \mathcal{C}_{2}(\mathcal{H})$ 
(Hilbert-Schmidt class), then 
$ \Phi_{A,B,\varphi}(X) \in 
\mathcal{C}_{2}(\mathcal{H})$
because $\mathcal{C}_{2}(\mathcal{H})$ 
is a Hilbert space, 
and $\Phi_{A,B,\varphi}$ can  be defined by the similar way to 
Theorem \ref{thm: double summation}.
To extend this into $X\in B(\mathcal{H})$, 
we shall consider 
Schur multiplier as follows:

\begin{definition}
[Schur multiplier, (cf. \cite{HK2003})]
\label{Schur multiplier}
When $\Phi_{A,B,\varphi}
(=\Phi_{A,B,\varphi}|_{\mathcal{C}_{1}(\mathcal{H})}):
X\mapsto \Phi_{A,B,\varphi}(X)$ 
gives rise to a bounded
transformation on the ideal 
$\mathcal{C}_{1}(\mathcal{H}) 
(\subset \mathcal{C}_{2}(\mathcal{H}))$
of trace class operators, $\varphi(s,t)$ is called 
a {\bf Schur multiplier} (relative to the pair $(A,B)$).
\end{definition}
 
The double operator integrals can be extended 
to $\mathcal{B(H)}$ by making use of the duality 
$\mathcal{B(H)}=\mathcal{C}_{1}(\mathcal{H})^{*}$ via
$$ (X,Y)\in \mathcal{C}_{1}(\mathcal{H})
\times \mathcal{B(H)}
\mapsto {\rm trace}(XY^{*})\in \mathbb{C}. $$
This proof is introduced in \cite{HK2003}.

\begin{thqt}[{\cite{HK2003, P1984, P1985}}]
\label{thqt:equivalence Schur}
For $\varphi \in L^{\infty}(\sigma(A)\times 
\sigma(B); \lambda\times \mu)$, the
following conditions are all equivalent:
\begin{itemize}
\item[{\rm (i)}] $\varphi$ is a Schur multiplier;
\item[{\rm (ii)}] whenever a measurable function 
$k:\sigma(A)\times \sigma(B) \to \mathbb{C}$ 
is the kernel of a trace class operator 
$L^{2}(\sigma(A); \lambda) \to 
L^{2}(\sigma(B); \mu)$, so is the product
$\varphi (s,t)k(s,t)$;
\item[{\rm (iii)}] one can find a finite measure space
$(\Omega, \sigma')$ and functions $\alpha\in 
L^{\infty}(\sigma(A)\times \Omega; \lambda\times 
\sigma')$,
$\beta\in 
L^{\infty}(\sigma(B)\times \Omega; \mu\times 
\sigma')$ such that 
\begin{equation}
\varphi(s,t)=\int_{\Omega} \alpha(s,x)
\beta(t,x)d\sigma'(x)
\label{eq: f(s,t)}
\end{equation}
for all $s\in \sigma(A)$, $t\in \sigma(B)$;
\item[{\rm (iv)}] 
one can find a measure space $(\Omega, \sigma')$
and measurable functions $\alpha$, $\beta$
on $\sigma(A)\times \Omega$,
$\sigma(B)\times \Omega$ respectively such that
the above \eqref{eq: f(s,t)} holds and 
$$ \left\| \int_{\Omega}|\alpha(\cdot, x)|^{2}
d\sigma'(x)\right\|_{L^{\infty}(\lambda)}
\left\| \int_{\Omega}|\beta(\cdot, x)|^{2}
d\sigma'(x)\right\|_{L^{\infty}(\lambda)}
<\infty. $$
\end{itemize}
\end{thqt}

An important result of this paper 
is to give a guarantee the 
perspective of representing functions $f$ 
of all operator means are Schur multiplier. 

\begin{theorem}\label{prop:Shur multiplier}
Let $f$ be a representing function 
of an operator mean. 
%
%
Then the perspective $\mathcal{P}_{f}$ of $f$
is a Schur multiplier.
\end{theorem}

\begin{proof}
%
%
By \eqref{eq:integral representation},
a perspective $\mathcal{P}_{f}$ of $f$ is
given as follows:
$$ \mathcal{P}_{f}(s,t)=sf\left(\frac{t}{s}\right)=
\int_{0}^{1}
[(1-\lambda)s^{-1}+\lambda t^{-1}]^{-1}
d\mu(\lambda). $$
By elementary computation, we have
\begin{align*}
\mathcal{P}_{f}(s,t)
& =
\int_{0}^{1}
[(1-\lambda)s^{-1}+\lambda t^{-1}]^{-1}
d\mu(\lambda)\\
& =
\int_{0}^{1}
\left( \int_{0}^{\infty}
e^{-(1-\lambda)s^{-1}x}
e^{-\lambda t^{-1}x}
dx\right)d\mu(\lambda).
\end{align*}
Putting $\alpha(s,x,\lambda)=e^{-(1-\lambda)s^{-1}x}$
and $\beta(t,x,\lambda)=e^{-\lambda t^{-1}x}$.
Then $\mathcal{P}_{f}$ can be represented as 
the form of \eqref{eq: f(s,t)}, 
and it is a Schur multiplier.
\end{proof}


\section{A generalization of the 
Aluthge transformation
in the operators case}

In this section, we shall give a definition of 
a generalized Aluthge transformation
by using the double operator integrals 
which is introduced in the previous section.

\begin{definition}
[Generalization of the Aluthge transformation]
\label{defi:generalized Aluthge transformation}
Let $T=U|T| \in B(\mathcal{H})$ with the 
spectral decomposition $|T|=\int_{\sigma(|T|)}sdE_{s}$.
For an operator mean $\frak{M}_{f}$,
a  {\bf generalization of the Aluthge transformation} 
$\Delta_{\frak{M}_{f}}(T)$ of $T$ respect to
an operator mean $\frak{M}_{f}$ is defined 
by
$$ \Delta_{\frak{M}_{f}}(T):=
\int_{\sigma(|T|)}\int_{\sigma(|T|)}\mathcal{P}_{f}(s,t)
dE_{s}UdE_{t}. $$
\end{definition}

By Theorem \ref{prop:Shur multiplier},
$\mathcal{P}_{f}$ is the Schur multiplier.
Hence the above double operator integrals is 
well-defined.
As in the similar discussion of Example 1, 
we obtain concrete forms of 
generalizations of the Aluthge transformation.
The following theorem is an extension of 
Theorem \ref{prop:matrix}.

\begin{theorem}\label{prop:another form}
Let $T=U|T|\in B(\mathcal{H})$, s.t., 
$|T|\in B(\mathcal{H})^{++}$.
For each operator mean $\frak{M}$, 
there exists a positive probability measure
$d\mu$ on $[0,1]$ such that
$$ \Delta_{\frak{M}}(T)=
\int_{0}^{1}\left( \int_{0}^{\infty}
e^{-x(1-\lambda) |T|^{-1}}U
e^{-x\lambda |T|^{-1}}dx \right)d\mu(\lambda). 
$$
\end{theorem}

\begin{proof}
For an operator mean $\frak{M}_{f}$, 
there exists a positive probability measure
$d\mu$ on $[0,1]$ such that
\begin{align*}
\mathcal{P}_{f}(s,t) & =
\int_{0}^{1}[(1-\lambda)s^{-1}+\lambda t^{-1}]^{-1}
d\mu(\lambda) \\
& =
\int_{0}^{1}\left(\int_{0}^{\infty}
e^{\{-x(1-\lambda)s^{-1}\}}
e^{\{-x\lambda t^{-1}\}}dx\right)d\mu(\lambda).
\end{align*}
Then by using Fubini-Tonelli's theorem, we 
have 
\begin{align*}
\Delta_{\frak{M}_{f}}(T) 
& =
\int_{\sigma(|T|)}\int_{\sigma(|T|)}
\mathcal{P}_{f}(s,t)
dE_{s}UdE_{t}\\
& =
\int_{\sigma(|T|)}\int_{\sigma(|T|)}
\left\{
\int_{0}^{1}\left(\int_{0}^{\infty}
e^{-x(1-\lambda)s^{-1}}
e^{-x\lambda t^{-1}}dx\right)d\mu(\lambda)
\right\}
dE_{s}UdE_{t}\\
& =
\int_{0}^{1}\left\{\int_{0}^{\infty}
\left(
\int_{\sigma(|T|)}e^{-x(1-\lambda)s^{-1}}
dE_{s}
\right)
U
\left(
\int_{\sigma(|T|)}
e^{-x\lambda t^{-1}}
dE_{t}
\right)dx\right\} d\mu(\lambda)\\
& =
\int_{0}^{1}\left(\int_{0}^{\infty}
e^{-x(1-\lambda)|T|^{-1}}
U
e^{-x\lambda |T|^{-1}}
dx\right)d\mu(\lambda).
\end{align*}
Hence the proof is completed.
\end{proof}

\begin{proposition}\label{prop: norm}
Let $T\in B(\mathcal{H})$ with the polar decomposition
$T=U|T|$. 
Then for any operator 
mean $\frak{M}_{f}$ and $\alpha\in \mathbb{C}$, 
the following statements hold.
\begin{itemize}
\item[(1)] $\Delta_{\frak{M}_{f}}(\alpha T)=\alpha
\Delta_{\frak{M}_{f}}(T)$,
\item[(2)] $\Delta_{\frak{M}_{f}}(V^{*} TV)=V^{*}
\Delta_{\frak{M}_{f}}(T)V$ for all unitary $V$,
\item[(3)] $\Delta_{\frak{M}_{f}}(T)-\alpha I=
\Phi_{|T|,|T|,\mathcal{P}_{f}}(U-\alpha |T|^{-1})$
if $|T|\in B(\mathcal{H})^{++}$,
\item[(4)] $ \| \Delta_{\frak{M}_{f}}(T)\|\leq \|T\| $,
where $\|\cdot \|$ means the spectral norm
on $B(\mathcal{H})$.
\end{itemize}
\end{proposition}

\begin{proof}
Let $|T|=\int_{\sigma(|T|)}sdE_{s}$ be
the spectral 
decomposition.

(1) Let $\alpha=re^{i\theta}$ $(r\geq 0)$ be a polar form of 
$\alpha\in \mathbb{C}$. Then
$\alpha T=(e^{i\theta}U)(\alpha |T|)$ is a polar
decomposition of $\alpha T$ moreover, 
$r|T|=\int_{\sigma(|T|)}rs dE_{s}$. Hence we have
\begin{align*}
\Delta_{\frak{M}_{f}}(\alpha T)
& =
\int_{\sigma(|T|)}\int_{\sigma(|T|)}
\mathcal{P}_{f}(rs,rt)dE_{s}(e^{i\theta}U)dE_{t} \\
& =
\int_{\sigma(|T|)}\int_{\sigma(|T|)}
re^{i\theta}\mathcal{P}_{f}(s,t)dE_{s}UdE_{t} 
 = 
\alpha \Delta_{\frak{M}_{f}}(T).
\end{align*}

(2) Let $V$ be unitary. Then
$V^{*}TV=V^{*}UV\cdot V^{*}|T|V$ is the polar decomposition.
Then we have
\begin{align*}
 \Delta_{\frak{M}_{f}}(V^{*}TV) &=
\int_{\sigma(|T|)}\int_{\sigma(|T|)}
\mathcal{P}_{f}(s,t) d(V^{*}E_{s}V) 
V^{*}UV d(V^{*}E_{t}V)\\
& =
V^{*} \left(\int_{\sigma(|T|)}\int_{\sigma(|T|)}
\mathcal{P}_{f}(s,t) dE_{s}UdE_{t}\right)V=
V^{*}\Delta_{\frak{M}_{f}}(T)V.
\end{align*}

(3) Since $dE_{s}$ is an orthogonal projection measure,
\begin{align*}
\Phi_{|T|,|T|,\mathcal{P}_{f}}(U-\alpha |T|^{-1}) 
& =
\int_{\sigma(|T|)}\int_{\sigma(|T|)} 
\mathcal{P}_{f}(s,t)dE_{s}(U-\alpha |T|^{-1})dE_{t}\\
& =
\int_{\sigma(|T|)}\int_{\sigma(|T|)} 
\mathcal{P}_{f}(s,t)dE_{s}UdE_{t}\\
& \hspace*{60pt}
-\alpha \int_{\sigma(|T|)}\int_{\sigma(|T|)} 
\mathcal{P}_{f}(s,t)dE_{s} |T|^{-1}dE_{t}\\
& =
\Delta_{\frak{M}_{f}}(T)-
\alpha \int_{\sigma(|T|)}\int_{\sigma(|T|)} 
\mathcal{P}_{f}(s,t)dE_{s} dE_{t}|T|^{-1}\\
& =
\Delta_{\frak{M}_{f}}(T)-
\alpha \int_{\sigma(|T|)} 
\mathcal{P}_{f}(s,s)dE_{s} |T|^{-1}\\
& =
\Delta_{\frak{M}_{f}}(T)-
\alpha \int_{\sigma(|T|)} 
s dE_{s} |T|^{-1}=
\Delta_{\frak{M}_{f}}(T)-\alpha I.
\end{align*}

(4) By Theorem \ref{prop:another form}, 
there exists a positive probability measure
$d\mu$ on $[0,1]$ such that
$$ \Delta_{\frak{M}_{f}}(T)=
\int_{0}^{1}\left(\int_{0}^{\infty}
e^{-x(1-\lambda) |T|^{-1}}U
e^{-x\lambda |T|^{-1}}dx\right) d\mu(\lambda). 
$$
Since $g(x)=e^{-cx^{-1}}$ is an increasing 
function on $x>0$ and for $c>0$, we have
$e^{-cA^{-1}}\leq e^{-c\|A\|^{-1}}I$. Hence 
we have
\begin{align*}
\|\Delta_{\frak{M}_{f}}(T)\|
& \leq 
\int_{0}^{1}\left(\int_{0}^{\infty}
\|e^{-x(1-\lambda) |T|^{-1}}U
e^{-x\lambda |T|^{-1}}\|dx\right)d\mu(\lambda)\\
& \leq 
\int_{0}^{1}\left(
\int_{0}^{\infty} e^{-x(1-\lambda) \|T\|^{-1}}
e^{-x\lambda \|T\|^{-1}}dx\right)d\mu(\lambda)
= \|T\|,
\end{align*}
where the last equality follows from 
$\int_{0}^{1}d\mu(\lambda)=1$.
\end{proof}

\begin{proposition}\label{prop:trace}
Let $T\in B(\mathcal{H})$,
and let 
$\frak{M}$ be an operator mean.
If $T\in C_{1}(\mathcal{H})$, 
then 
$$ {\rm trace}(\Delta_{\frak{M}}(T))=
{\rm trace}(T). $$
\end{proposition}

\begin{proof}
It follows from Theorem \ref{prop:another form}.
\end{proof}

We notice that it is known that
if  $\frak{M}$ is a weighted geometric mean 
(i.e., $\Delta_{\frak{M}}(T)$ is the 
generalized Aluthge transformation), 
then
$\sigma(\Delta_{\frak{M}}(T))=\sigma(T)$ holds for 
all $T\in B(\mathcal{H})$ \cite{H1997}.
However, there is a counterexample 
for this equation when 
$\frak{M}$ is an arithmetic mean 
\cite{LLY2014}.

\section{Iteration}

In this section, we shall consider 
iteration of $\Delta_{\frak{M}}$.
For each natural number $n$, 
define $\Delta^n_{\frak{M}}(T):=
\Delta_{\frak{M}}(\Delta^{n-1}_{\frak{M}}(T))$
and $\Delta_{\frak{M}}^{0}(T):=T$ for $T\in B(\mathcal{H})$.
It is known that iteration of the Aluthge 
transformation 
(i.e., $\frak{M}$ is a geometric mean)
has been considered by 
many authors, for example, 
(i) a sequence 
$\{\Delta^{n}_{\frak{M}}(T)\}_{n=0}^{\infty}$
converges to a normal matrix if 
$T\in \mathcal{M}_{m}$ \cite{AY2003, APS2011},
(ii) there exists an operator $T\in B(\mathcal{H})$
such that  a sequence 
$\{\Delta^{n}_{\frak{M}}(T)\}_{n=0}^{\infty}$
does not converge in the week operator topology
\cite{CJL2005},
(iii) for each $T\in B(\mathcal{H})$,
$\lim_{n\to\infty}\| \Delta^{n}_{\frak{M}}(T)\|
=r(T)$, where $r(T)$ is the spectral radius of
$T$ \cite{Y2002}.


\begin{theorem}\label{thm:iteration-finite}
Let $T\in \mathcal{M}_{m}$ 
be invertible with the polar decomposition
$T=U|T|$.
Let $\frak{A}$ be a 
non-weighted arithmetic mean.
Then a sequence $\{\Delta^n_{\frak{A}}(T)\}$ 
converges to a normal matrix $N$ 
such that
${\rm trace}(T)={\rm trace}(N)$ and
${\rm trace}(|T|)={\rm trace}(|N|)$.
\end{theorem}

The iteration of mean 
transformation has been considered 
in \cite{arXivCCM2018}. However, 
in \cite{arXivCCM2018},
the authors have considered only
rank one operators or they required some
conditions.
We notice that if $T$ is invertible,
then $\Delta_{\frak{A}}(T)$ coincides with 
the mean transformation $\hat{T}$ of $T$ 
(see Example 1).

To prove Theorem \ref{thm:iteration-finite}, 
we use a useful formula which is 
shown in \cite{arXivCCM2018}.

\begin{thqt}[\cite{arXivCCM2018}]
\label{thqt:formula of iteration}
Let $T\in \mathcal{B(H)}$ and suppose that
$\ker(T^{*})\subseteq \ker(T)$. Let 
$T=U|T|$ be the canonical polar decomposition of $T$,
and let $n\in \mathbb{N}$. 
Then
the polar decomposition of $\hat{T}^{(n)}$
is
$$ \hat{T}^{(n)}=
U\cdot \frac{1}{2^{n}} \sum_{k=0}^{n}
\begin{pmatrix} n \\ k \end{pmatrix} 
(U^{*})^{k}|T|U^{k}, $$
where $\hat{T}^{(n)}:=\hat{(T^{(n-1)})}$ and 
$\hat{T}^{(1)}:=\hat{T}$.
\end{thqt}

\begin{proof}
[Proof of Theorem \ref{thm:iteration-finite}]
Since $T$ is invertible,
$\Delta_{\frak{A}}(T)$ coincides with 
the mean transformation $\hat{T}$.
Moreover $\Delta_{\frak{A}}(T)=U\cdot 
\frac{1}{2}(|T|+U^{*}|T|U)$ is the polar decomposition
by Theorem \ref{thqt:formula of iteration}.
We notice that since $T$ is invertible,
$U$ and $|T|$ are invertible, and hence 
$\Delta_{\frak{A}}(T)$ is invertible.
Therefore for each natural number $n$,
$\Delta^{(n)}_{\frak{A}}(T)=\hat{T}^{(n)}$ holds, 
and it is invertible.
By Theorem \ref{thqt:formula of iteration}, 
we only prove that 
$\{|\Delta^n_{\frak{A}}(T)|\}$ converges
to a positive matrix.
Since $U$ is unitary, we can diagonalize $U$ as 
$ U=V^{*} DV$ with a unitary matrix $V$ and 
$D={\rm diag}(e^{\theta_{1}\sqrt{-1}},...,
e^{\theta_{m}\sqrt{-1}})$.
Then by Proposition \ref{prop: norm} (2), 
%
%
$\Delta_{\frak{A}}^{n}(T)=
V^{*} \Delta_{\frak{A}}^{n} (DV |T|V^{*} )V$, and 
\begin{align*}
|\Delta^n_{\frak{A}}(T)| & =
V^{*} |\Delta_{\frak{A}}^{n} (DV |T|V^{*} )|V \\
& =
V^{*} \left\{ \frac{1}{2^{n}}
\sum_{k=0}^{n}\begin{pmatrix} n \\ k \end{pmatrix}
{(D^{*})}^{k}V|T|V^{*}D^{k} \right\}  V.
\end{align*}
Let $V|T|V^{*}=P$. Then 
${D^{*}}^{k}V|T| V^{*}D^{k}=
[e^{k(\theta_{j}-\theta_{i})\sqrt{-1}}]\circ P$,
where 
$[e^{k(\theta_{j}-\theta_{i})\sqrt{-1}}]\in
\mathcal{M}_{m}$ with the $(i,j)$-entry 
$e^{k(\theta_{j}-\theta_{i})\sqrt{-1}}$, 
and $\circ$ means the Hadamard product.
Hence
\begin{align*}
\frac{1}{2^{n}}
\sum_{k=0}^{n}\begin{pmatrix} n \\ k \end{pmatrix}
{D^{*}}^{k}V|T|V^{*}D^{k} 
& =
\frac{1}{2^{n}} \left[
\sum_{k=0}^{n} \begin{pmatrix} n \\ k \end{pmatrix}
e^{k(\theta_{j}-\theta_{i})\sqrt{-1}}\right]\circ P\\
& =
\frac{1}{2^{n}} \left[
(1+e^{(\theta_{j}-\theta_{i})\sqrt{-1}})^{n}\right]\circ P\\
& \hspace*{3cm} \mbox{(by the binomial expansion)}\\
& =
\left[
\left(\frac{1+e^{(\theta_{j}-\theta_{i})\sqrt{-1}}}
{2}\right)^{n}\right]\circ P.
\end{align*}
We notice that
$ \left| 
\frac{1+e^{(\theta_{j}-\theta_{i})\sqrt{-1}}}{2}\right|<1$ 
if $\theta_{j}\neq \theta_{i}+2k\pi$ for all 
integers $k$, and
$  \frac{1+e^{(\theta_{j}-\theta_{i})\sqrt{-1}}}{2}=1$ 
if $\theta_{j}= \theta_{i}+2k\pi$ for some integers $k$.
Hence there exists a matrix $P_{0}$ such that
$$ \lim_{n\to \infty} |\Delta^n_{\frak{A}}(T)|=
\lim_{n\to \infty}
V^{*} \left(\left[
\left(
\frac{1+e^{(\theta_{j}-\theta_{i})\sqrt{-1}}}{2}
\right)^{n}
\right]\circ P\right) V=P_{0}, $$
and we have $\lim_{n\to \infty}
\Delta^n_{\frak{A}}(T)=
\lim_{n\to \infty}U|\Delta^n_{\frak{A}}(T)|=
UP_{0}=N$. Moreover since
$UP_{0}=\Delta_{\frak{A}}(UP_{0})
=\frac{UP_{0}+P_{0}U}{2}$, 
$UP_{0}=P_{0}U$ holds, i.e.,  
$N=UP_{0}$ is a normal matrix.

By Proposition \ref{prop:trace} and the above,
we have 
${\rm trace}(T)={\rm trace}(\Delta^n_{\frak{A}}(T))
={\rm trace}(N)$ and
${\rm trace}(|T|)=
{\rm trace}(|\Delta^n_{\frak{A}}(T))|
={\rm trace}(|N|)$ 
for all $n=0,1,2,...$.
\end{proof}

\begin{theorem}\label{cor:converge3}
Let $\frak{M}_{f}$ be an operator mean
whose representing function satisfies 
$\lambda=f'(1)\in (0,1)$.
Then there exists an operator $T\in B(\mathcal{H})$ 
such that
$ \{\Delta_{\frak{M}_{f}}^{n}(T)\} $
does not converge in the week operator topology.
\end{theorem}

To prove Theorem \ref{cor:converge3},
we prepare the following discussion.
It is a modification of \cite[Corollary 3.3]{CJL2005}:
Let $\alpha:=(\alpha_{0},\alpha_{1},\alpha_{2},...)$
be a sequence in $\ell^{\infty}$, and let 
$W_{\alpha}$ be a weighted unilateral shift on
$\ell^{2}$ with 
a weight sequence $\alpha$, i.e., 
$$ W_{\alpha}e_{n}=\alpha_{n}e_{n+1}, $$
where $\{e_{n}\}$ be the canonical basis 
of $\ell^{2}$. In what follows,
$\alpha_{n}>0$ for all $n=0,1,2,...$.

\begin{lemma}\label{lem:concrete form}
Let $f,g$ be representing functions of 
weighted arithmetic and harmonic means 
with a weight $\lambda\in [0,1]$, respectively,
and let
$\alpha=(\alpha_{0},\alpha_{1},\alpha_{2},...)$
be a sequence in $\ell^{\infty}$.
Then 
\begin{align*}
\mathcal{P}_{f}(\alpha_{1}^{(n)}, \alpha_{0}^{(n)})
& =
\sum_{j=0}^{n+1}\begin{pmatrix} n+1 \\ j 
\end{pmatrix} \lambda^{n+1-j}(1-\lambda)^{j}\alpha_{j},\\
\mathcal{P}_{g}(\beta_{1}^{(n)}, \beta_{0}^{(n)})
& =
\left[\sum_{j=0}^{n+1}\begin{pmatrix} n+1 \\ j 
\end{pmatrix} \lambda^{n+1-j}
(1-\lambda)^{j}\alpha_{j}^{-1}\right]^{-1},
\end{align*}
where 
$\alpha^{(n)}:=(\alpha_{0}^{(n)},\alpha_{1}^{(n)},...)$ and 
$\beta^{(n)}:=(\beta_{0}^{(n)},\beta_{1}^{(n)},...)$
are 
\begin{align*}
\alpha_{k}^{(n)} & = \mathcal{P}_{f}(\alpha_{k+1}^{(n-1)},
\alpha_{k}^{(n-1)}),\\
\beta_{k}^{(n)} & = \mathcal{P}_{g}(\beta_{k+1}^{(n-1)},
\beta_{k}^{(n-1)})
\end{align*}
and
$
\alpha_{k}^{(0)}=\beta_{k}^{(0)}=\alpha_{k}.
$
\end{lemma}

\begin{proof}
The proof follows from the mathematical induction.
\end{proof}

\begin{lemma}\label{prop:converge0}
Let $\frak{M}_{f}$ be an operator mean, and
let $W_{\alpha}$ be a weighted unilateral shift
with a weight sequence $\alpha$. 
Then 
$\Delta^{n}_{\frak{M}}(W_{\alpha})$ 
is also a weighted unilateral shift $W_{\alpha'}$
with a weight sequence
$ \alpha'=(\alpha_{0}',\alpha_{1}',...), $
where 
$ \alpha_{n}'=\mathcal{P}_{f}(\alpha_{n+1},
\alpha_{n}). $ $(n=0,1,2,...)$.
\end{lemma}

\begin{proof}
Let $\{e_{n}\}$ be the canonical basis of $\ell^{2}$.
Then $W_{a}$ can be represented to
$$ W_{\alpha}=
\begin{pmatrix} 0 \\
\alpha_{0} & 0  \\
 & \alpha_{1} & 0  \\
& & \alpha_{2} & 0 \\
& & & \ddots & \ddots 
\end{pmatrix}
$$ 
respect to $\{e_{n}\}$, and 
we have the spectral decomposition 
$|W_{\alpha}|=\sum_{n=0}^{\infty}
\alpha_{n}P_{n}$, where 
every $P_{n}=(p_{ij})$ is a projection satisfying
$$p_{ij}=\left\{ \begin{array}{ll}
1 & \mbox{$(i=j=n)$}  \\
0 & \mbox{(otherwise)}
\end{array}\right. .$$
Also assume that $W_{\alpha}=U|W_{\alpha}|$
is the polar decomposition. Then $U$ is a 
unilateral shift.
%
%
Thus for $n=0,1,...$, 
%
%
\begin{align*}
\Delta_{\frak{M}_{f}}(W_{\alpha})e_{n}
& =
\sum_{i,j=0}^{\infty}\mathcal{P}_{f}
(\alpha_{i},\alpha_{j})P_{i}UP_{j}e_{n} \\
& =
\sum_{i=0}^{\infty}\mathcal{P}_{f}
(\alpha_{i},\alpha_{n})P_{i}UP_{n}e_{n} \\
& =
\sum_{i=0}^{\infty}\mathcal{P}_{f}
(\alpha_{i},\alpha_{n})P_{i}Ue_{n} \\
& =
\sum_{i=0}^{\infty}\mathcal{P}_{f}
(\alpha_{i},\alpha_{n})P_{i}e_{n+1} \\
& =
\mathcal{P}_{f}
(\alpha_{n+1},\alpha_{n})P_{n+1}e_{n+1} 
 =
\mathcal{P}_{f}
(\alpha_{n+1},\alpha_{n})e_{n+1}, 
\end{align*}
i.e., $\Delta_{\frak{M}_{f}}(W_{\alpha})$ is also
a weighted unilateral shift with a weight sequence
$\alpha'$.
\end{proof}

\begin{proposition}\label{prop:converge1}
Let $\frak{A},\frak{H}$ be $\lambda$-
weighted arithmetic and harmonic means with 
a weight $\lambda\in (0,1)$, respectively.
Suppose that $a$ and $b$ are any distinct positive 
real numbers. Let $W_{\alpha}$ be a 
unilateral weighted shift whose weights are either 
$a$ or $b$. Suppose that only finitely many weights
of $W_{\alpha}$ are equal to $a$. Then 
the sequences of the first weights of
$\Delta^{n}_{\frak{A}}(T)$ and 
$\Delta^{n}_{\frak{H}}(T)$ converge to $b$.
\end{proposition}

\begin{proof}

By Lemmas \ref{lem:concrete form} and 
\ref{prop:converge0}, the first weights of 
$\Delta^{n}_{\frak{A}}(W_{\alpha})$ and 
$\Delta^{n}_{\frak{H}}(W_{\alpha})$ are
\begin{align*}
\alpha_{0}^{(n)}
& =
\sum_{j=0}^{n}\begin{pmatrix} n \\ j 
\end{pmatrix} \lambda^{n-j}(1-\lambda)^{j}\alpha_{j},\\
\beta_{0}^{(n)}
& =
\left[\sum_{j=0}^{n}\begin{pmatrix} n \\ j 
\end{pmatrix} \lambda^{n-j}
(1-\lambda)^{j}\alpha_{j}^{-1}\right]^{-1},
\end{align*}
respectively. Next, let $p$ be the largest number satisfying 
$\alpha_{p}=a$. Then for $n>p$, we have
\begin{align*}
\alpha_{0}^{(n)}
& =
\sum_{j=0}^{n}\begin{pmatrix} n \\ j 
\end{pmatrix} \lambda^{n-j}(1-\lambda)^{j}\alpha_{j}\\
& =
\sum_{j=0}^{p}\begin{pmatrix} n \\ j 
\end{pmatrix} \lambda^{n-j}(1-\lambda)^{j}a_{j}+
b\sum_{j=p+1}^{n}\begin{pmatrix} n \\ j 
\end{pmatrix} \lambda^{n-j}(1-\lambda)^{j}.
\end{align*}
Here we shall show 
$$\lim_{n\to\infty}
\sum_{j=0}^{p}\begin{pmatrix} n \\ j 
\end{pmatrix} \lambda^{n-j}(1-\lambda)^{j}a_{j}=0.$$
Let $m=\max\{a,b\}$ and $M=\max_{0\leq j\leq p}
\frac{1}{j!}(\frac{1-\lambda}{\lambda})^{j}$.
Then 
\begin{align*}
0 & \leq 
\sum_{j=0}^{p}\begin{pmatrix} n \\ j 
\end{pmatrix} \lambda^{n-j}(1-\lambda)^{j}a_{j}\\
& \leq 
m \sum_{j=0}^{p}\begin{pmatrix} n \\ j 
\end{pmatrix} \lambda^{n-j}(1-\lambda)^{j}\\
& \leq 
m \sum_{j=0}^{p}
n^{j}\lambda^{n}\frac{1}{j!}
\left(\frac{1-\lambda}{\lambda}\right)^{j}\\
& \leq
mM \sum_{j=0}^{p}
n^{p}\lambda^{n} 
 =
mM(p+1) n^{p}\lambda^{n}.
\end{align*}
Since $\lim_{n\to\infty} n^{p}\lambda^{n}=0$
(by l'Hospital's rule), we have
\begin{equation}
\lim_{n\to\infty}\sum_{j=0}^{p}
\begin{pmatrix} n \\ j 
\end{pmatrix} \lambda^{n-j}(1-\lambda)^{j}a_{j}=0. 
\label{eq: *}
\end{equation}
Hence we have 
\begin{align*}
\lim_{n\to \infty} \alpha_{0}^{(n)}
& =
\lim_{n\to \infty} 
\sum_{j=0}^{p}\begin{pmatrix} n \\ j 
\end{pmatrix} \lambda^{n-j}(1-\lambda)^{j}a_{j}\\
&\hspace*{1cm} +
\lim_{n\to \infty} 
b\sum_{j=p+1}^{n}\begin{pmatrix} n \\ j 
\end{pmatrix} \lambda^{n-j}(1-\lambda)^{j}=b,
\end{align*}
since $\sum_{j=0}^{n}\begin{pmatrix} n \\ j 
\end{pmatrix}
\lambda^{n-j}(1-\lambda)^{j}=1$ and 
\eqref{eq: *}.
$\lim_{n\to\infty}\beta_{0}^{(n)}=b$
can be proven by the 
same way, too.
\end{proof}

\begin{proposition}\label{prop:converge2}
Suppose $a$ and $b$ are any distinct positive real
numbers. 
Then there is a unilateral weighted shift
$W_{\alpha}$ with a weight sequence $\alpha$ 
such that both sequences of the first weights of
$\Delta_{\frak{A}}^{n}(W_{\alpha})$ and 
$\Delta_{\frak{H}}^{n}(W_{\alpha})$ have subsequences 
which are converging to $a$ and $b$.
\end{proposition}

\begin{proof}
Suppose that $\alpha=(a,b,b,...)$. 
By Proposition \ref{prop:converge1}, 
there exists a natural number $n_{1}$ such that
$$ |\alpha_{0}^{(n_{1})}-b|<\frac{1}{2}\quad
\text{and}\quad
|\beta_{0}^{(n_{1})}-b|<\frac{1}{2}. $$
Suppose that 
$$\alpha=(
\overbrace{a,b,b,...,b}^{n_{1}},a,a,...).$$ 
By Proposition \ref{prop:converge1}, 
there exists a natural number $n_{2}$ such that $n_{1}<n_{2}$,
$$ |\alpha_{0}^{(n_{2})}-a|<\frac{1}{2^2}\quad
\text{and}\quad
|\beta_{0}^{(n_{2})}-a|<\frac{1}{2^2}. $$
Suppose that 
$$\alpha=(
\underbrace{
\overbrace{a,b,b,...,b}^{n_{1}},a,a,...,a}_{n_{2}},b,b,...).$$ 
By Proposition \ref{prop:converge1}, 
there exists a natural number $n_{3}$ such that  $n_{1}<n_{2}<n_{3}$,
$$ |\alpha_{0}^{(n_{3})}-b|<\frac{1}{2^3}\quad
\text{and}\quad
|\beta_{0}^{(n_{3})}-b|<\frac{1}{2^3}. $$

Repeating this process, for each
natural number $k$, there exist natural numbers $n_{k}$ 
and $n'_{k}$ such that
\begin{align*}
& 
|\alpha_{0}^{(n_{k})}-a|<\frac{1}{2^k},\quad
|\beta_{0}^{(n_{k})}-a|<\frac{1}{2^k}\\
& 
|\alpha_{0}^{(n'_{k})}-b|<\frac{1}{2^k},\quad
|\beta_{0}^{(n'_{k})}-b|<\frac{1}{2^k}.
\end{align*}
Hence there are subsequences of
the first weights of 
$\Delta_{\frak{A}}^{n}(W_{\alpha})$ and 
$\Delta_{\frak{H}}^{n}(W_{\alpha})$ which converge to $a$ and $b$.
\end{proof}

\begin{proof}[Proof of Theorem \ref{cor:converge3}]
Let $T:=W_{\alpha}$ be a weighted unilateral shift 
defined in Proposition \ref{prop:converge2}.
Define a sequence of the first weight of 
$\Delta_{\frak{M}_f}^{n}(T)$ by 
$\{\gamma_{0}^{(n)}\}$.
By $\lambda=f'(1)\in (0,1)$ and an inequality 
\eqref{eq:condition}, we have 
$$ \beta_{0}^{(n)}\leq \gamma_{0}^{(n)}\leq 
\alpha_{0}^{(n)} \quad\mbox{for $n=0,1,2,...$}$$
By Proposition \ref{prop:converge2},
there exist subsequences of the first weights of
$\Delta_{\frak{A}}^{n}(T)$ and 
$\Delta_{\frak{H}}^{n}(T)$
which converge to $a$ 
in the same choice.
Hence under the same choice, 
a subsequence of 
$\{\gamma_{0}^{(n)}\}$ converges to $a$.
On the other hand, we can choose a subsequence
of $\{\gamma_{0}^{(n)}\}$ which converges to $b$.
Hence $ \{\Delta_{\frak{M}}^{n}(T)\} $
does not converge in the week operator topology.
\end{proof}

\section{Numerical ranges of  
generalizations of the Aluthge transformation}

In this section, we shall show inclusion 
relations among numerical ranges of 
 generalizations of the Aluthge transformation.
For each operator $T\in B(\mathcal{H})$, 
the {\bf numerical range} $W(T)$ of $T$ is 
defined by
$$ W(T)=\{\langle Tx,x\rangle |\ 
x\in \mathcal{H} \mbox{ is a unit vector}\}.$$
It is well known that the numerical range 
of an operator is a convex subset in 
$\mathbb{C}$ \cite{H1919, T1918}.
Let $\frak{M}$ be the geometric mean
i.e., $\Delta_{\frak{M}}(T)$ is the Aluthge 
transformation.
Then the following assertions are known:
(i) $\overline{W(\Delta_{\frak{M}}(T))}\subseteq
\overline{W(T)}$, where
$\overline{X}$ is a closure of a set $X$ 
\cite{W2002, Y2002-1},
(ii) $co \sigma(T)=\cap 
W(\Delta_{\frak{M}}^{n}(T))$ for any matrix $T$,
where $co X$ is a convex hull of a set $X$ 
\cite{A2004}.
In this section, we shall give a 
relation among $\overline{W(\Delta_{\frak{M}}(T))}$ 
respect to some operator means.
Here we shall consider a relation ``$\preceq$''
between two operator means.
It is defined as follows.
Let $\frak{M}_{f}$ and $\frak{N}_{g}$ be 
operator means.
$\frak{M}_{f}\preceq \frak{N}_{g}$ if and only if  
for any natural number $m$ and $s_{i}>0$ $(i=1,2,...,m)$,
$\displaystyle  \left[\frac{\mathcal{P}_{f}(s_{i},s_{j})}
{\mathcal{P}_{g}(s_{i},s_{j})}\right]\in \mathcal{M}_{m} $
is a positive semi-definite.
Let $\frak{A}$, $\frak{L}$, $\frak{G}$ and 
$\frak{H}$ be non-weighted arithmetic, logarithmic,
geometric and harmonic means, respectively.
It is known that
\begin{equation}
\frak{H}\preceq \frak{G}\preceq
\frak{L}\preceq \frak{A} 
\label{eq: positive definiteness of functions}
\end{equation}
(cf. \cite{AN2017, HK2003, K2014}).

\begin{theorem}\label{thm:numerical range}
Let $T\in B(\mathcal{H})$, and let 
$\frak{M}_{f}, \frak{N}_{g}$ be operator means.
Then the following hold.
\begin{itemize}
\item[(1)] 
If $\frak{M}_{f}\preceq \frak{N}_{g}$, then 
 $\overline{W(\Delta_{\frak{M}_{f}}(T))}\subseteq
\overline{W(\Delta_{\frak{N}_{g}}(T))}$.
\item[(2)]
If $\frak{M}_{f}  \preceq \frak{A}$, then
$\overline{W(\Delta_{\frak{M}_{f}}(T))}\subseteq
\overline{W(T)}$,
where $\frak{A}$ is the arithmetic mean.
\end{itemize}
\end{theorem}

If $\ker (T)=\{0\}$, then $\Delta_{\frak{A}}(T)=\hat{T}$,
and by \eqref{eq: positive definiteness of functions},
Theorem \ref{thm:numerical range} is an extension of 
\cite[Theorem 3.1, Corollary 3.3]{arXivCCM2018}.

To prove Theorem \ref{thm:numerical range},
we shall use the following results.
A norm $|\!|\!| \cdot |\!|\!|$ is called 
a {\bf unitarily invariant norm} if 
$|\!|\!| UXV |\!|\!|=|\!|\!| X |\!|\!|$ holds for 
all unitary $U,V$ and $X\in B(\mathcal{H})$.


\begin{thqt}[\cite{H1965, SW1968}]
\label{thmqt: characterization of numerical range}
Let $T\in B(\mathcal{H})$. Then 
$$ \overline{W(T)}=\bigcap_{\lambda\in \mathbb{C}}
\{\mu\in \mathbb{C}\  |\ 
|\mu-\lambda|\leq \|T-\lambda I\|\}.$$
\end{thqt}


\begin{thqt}[\cite{HK1999, HK2003}]
\label{thmqt: norm inequality}
Let $\frak{M}_{f}, \frak{N}_{g}$ be operator means,
and let $A,B\in B(\mathcal{H})^{+}$.
Then $\frak{M}_{f}\preceq \frak{N}_{g}$ holds
if and only if 
$$ |\!|\!| \Phi_{A,B,\mathcal{P}_{f}}(X)|\!|\!|\leq 
|\!|\!| \Phi_{A,B,\mathcal{P}_{g}}(X) |\!|\!|$$
holds for all $X\in B(\mathcal{H})$ 
and any unitarily invariant norm $|\!|\!| \cdot |\!|\!|$.
\end{thqt}

Using the above results, we shall show 
Theorem \ref{thm:numerical range}.

\begin{proof}[Proof of Theorem 
\ref{thm:numerical range}]
(1) For $\varepsilon >0$, $|T|_{\varepsilon}:=
|T|+\varepsilon I$ is positive invertible, and 
$|T|_{\varepsilon}\searrow |T|$ as $\varepsilon 
\searrow 0$. By considering this fact,
we may assume that $|T|$ is 
invertible.

By Proposition \ref{prop: norm} (3), 
for $\lambda\in \mathbb{C}$, 
$$\Phi_{|T|,|T|,\mathcal{P}_{f}}(U-\lambda |T|^{-1})=
\Delta_{\frak{M}_{f}}(T)-\lambda I.
$$
Then by Theorem \ref{thmqt: norm inequality},
we have
\begin{align*}
\| \Delta_{\frak{M}_{f}}(T)-\lambda I\|& =
\| \Phi_{|T|,|T|,\mathcal{P}_{f}}(U-\lambda |T|^{-1})\|\\
& \leq 
\| \Phi_{|T|,|T|,\mathcal{P}_{g}}(U-\lambda |T|^{-1})\|=
\| \Delta_{\frak{N}_{g}}(T)-\lambda I\|
\end{align*}
for all $\lambda\in \mathbb{C}$.
Hence by Theorem 
\ref{thmqt: characterization of numerical range}, 
we have $\overline{W(\Delta_{\frak{M}_{f}}(T))}\subseteq
\overline{W(\Delta_{\frak{N}_{g}}(T))}$.

(2) Let $h(x)=1-\lambda+\lambda x$ be 
a representing function of $\frak{A}$.
By (1), we have
$\overline{W(\Delta_{\frak{M}_{f}}(T))}\subseteq
\overline{W(\Delta_{\frak{A}}(T))}$.
Hence we only prove 
$\overline{W(\Delta_{\frak{A}}(T))}\subseteq
\overline{W(T)}$.
Let $x\in \mathcal{H}$ be a unit vector.
Since
$$ \langle |T|Ux,x\rangle=
\langle TUx,Ux\rangle =
\| Ux\|^{2}\langle T \frac{Ux}{\|Ux\|}, 
\frac{Ux}{\|Ux\|}\rangle,$$
$$\langle |T|Ux,x\rangle\in co \{W(T)\cup 
\{0\}\}. $$
If $\{0\}\subset \ker(U)=\ker(T)$, 
then $0\in W(T)$, and
$\langle |T|Ux,x\rangle\in W(T)$.
If $\ker(U)= \{0\}$, then $U$ is isometry, and
$\langle |T|Ux,x\rangle\in W(T)$.
By the similar discussion, 
we have $\langle U^*UU|T|x,x\rangle\in W(T)$.
Hence we have
\begin{align*}
\overline{W(\Delta_{\frak{A}}(T))} & =
\overline{W((1-\lambda) |T|U+\lambda U^*U U|T|)}\\
& \subseteq 
(1-\lambda)\overline{W(|T|U)}+
\lambda \overline{W(U^*U U|T|)}
 \subseteq \overline{W(T)}.
\end{align*}
\end{proof}

\end{document}